\documentclass[11pt]{amsart}
\usepackage{amssymb, graphicx, color}
\usepackage{amsfonts}
\usepackage{amsthm}
\usepackage{enumerate}
\usepackage[mathscr]{eucal}
\usepackage{graphicx}
\usepackage[pdftex, bookmarksnumbered, bookmarksopen, colorlinks, citecolor=blue, linkcolor=blue]{hyperref}

\allowdisplaybreaks

\newtheorem{thm}{Theorem}[section]
\newtheorem{prop}[thm]{Proposition}
\newtheorem{lemma}[thm]{Lemma}
\theoremstyle{definition}%定理展示的样式
\newtheorem{definition}[thm]{Definition}

\theoremstyle{remark}
\newtheorem{remark}[thm]{Remark}

\newcommand{\R}{{\mathbb R}}
\newcommand{\N}{{\mathbb N}}

\newcommand{\M}{{\mathcal M}}

\newcommand{\bs}{\begin{split}}
\newcommand{\es}{\end{split}}

\begin{document}
\setcounter{page}{1}	
\title[Maximal Calder\'on-Zygmund operators]
{Noncommutative maximal strong $L_p$ estimates of Calder\'on-Zygmund operators}	
\author{Guixiang Hong}
\address{
	School of Mathematics and Statistics\\
	Wuhan University\\
	Wuhan 430072\\
	China}

\email{guixiang.hong@whu.edu.cn}
	
\author{Xudong Lai}
\address{
Institute for Advanced Study in Mathematics\\
Harbin Institute of Technology\\
Harbin
150001\\
China}

\email{xudonglai@hit.edu.cn}
 
		\author{Samya Kumar Ray}
\address{Statistics and Mathematics Unit, Indian Statistical Institute\\
203, B. T. Road, Kolkata, 700108\\ 
India}

\email{samyaray7777@gmail.com}
\author{Bang Xu}
\address{
Department of Mathematical Sciences\\
	Seoul National University\\
	Seoul 08826\\
	Republic of Korea}

\email{bangxu@whu.edu.cn}
\thanks{G. Hong and B. Xu were supported by National Natural Science Foundation of China (No. 12071355). X. Lai was supported by National Natural Science Foundation of China (No. 12271124 and No. 11801118), Heilongjiang Provincial Natural Science Foundation of China (YQ2022A005), Fundamental Research Funds for the Central Universities (No. FRFCU5710050121). S. Ray was supported by the DST-INSPIRE Faculty Fellowship DST/INSPIRE/04/2020/001132}



	\begin{abstract}
		In this paper, we obtain the desired noncommutative maximal inequalities of the truncated Calder\'on-Zygmund operators of non-convolution type acting on operator-valued $L_p$-functions for all $1<p<\infty$, answering a question left open in the previous work \cite{HLX}. 
	\end{abstract}

		\maketitle
		
\section{introduction}

This paper is devoted to the study of maximal singular integral theory in noncommutative analysis. The starting point of this problem is based on the boundedness theory of noncommutative Calder\'on-Zygmund (abbrieviated as CZ) operators (cf. \cite{M, JP1}), and this theory immediately showed great generality in solving problems in other fields (see e.g. \cite{GJP,JMP1,JMP2}).

Let us first set some notation. Recall that a function $k:\;\mathbb{R}^d \times \mathbb{R}^d \setminus \{(x,x):\;x\in\mathbb R^d\}\rightarrow\mathbb C$ is called a {\emph{standard}} kernel if it satisfies for some constant $C>0$ the size condition
\begin{align}\label{1}
	|k(x,y)| \ \leq \
	\frac{C}{|x-y|^d}
\end{align}
and for some $\gamma\in(0,1)$ the Lipschitz regularity condition 
\begin{align}\label{232}
	\begin{array}{rcl} \big| k(x,y) - k(x,z) \big|+\big| k(y,x) - k(z,x) \big|& \leq &
		\displaystyle \frac{C|y-z|^\gamma}{|x-y|^{d+\gamma}}   \hskip1pt
	\end{array}
\end{align}
whenever $|x-y|\geq2|y-z|$. An operator $T$ is said to be a {\it standard} CZ operator if it is associated with a {\emph{standard}} kernel $k$, that is for any compactly supported bounded function $f$,
\begin{equation}\label{30}
	Tf(x) = \int_{\mathbb{R}^d} k(x,y) f(y) \, dy,\ \  \ x \notin{\mathrm{supp}}  \hskip1pt f
\end{equation}
%, where $\overrightarrow{\mathrm{supp}}  \hskip1pt f=\mathrm{supp}\hskip1pt \|f\|_{\mathcal{M}}$; 
and $T$ admits a bounded extension on $L_2(\mathbb R^d)$.

%We formulate our results for maximal Calder\'on-Zygmund operators. Let $\mathcal{M}$  Let $T$ be a Calder\'on-Zygmund operator (abbrieviated as CZO) associated with a {\emph{standard}} kernel $k: \mathbb{R}^d \times \mathbb{R}^d \setminus \{(x,x):\;x\in\mathbb R^d\} \to \mathbb{C}$. This means that for , we have

%where ; moreover,  Recall that a kernel $k$ is called {\emph{standard}}
%if it satisfies the size condition
% where and in the sequel, $C$ is a positive numerical constant; and the $\gamma$-
Let $(\M,\tau)$ be a semifinite von Neumann algebra equipped with a normal semifinite faithful (abbrieviated as \emph{n.s.f}) trace. The noncommutative $L_p$-spaces associated with pair $(\M,\tau)$ are denoted by $L_p(\M)$. It is trivial that a {\it standard} CZ operator $T$ admits a bounded extension to noncommutative space $L_2(L_{\infty}(\mathbb{R}^{d})\overline{\otimes}\mathcal{M})$, which is still denoted by $T$ without causing any confusion. The extension to $L_p(L_{\infty}(\mathbb{R}^{d})\overline{\otimes}\mathcal{M})$ for $1<p\neq2<\infty$ was first justified by vector-valued theory (see e.g. \cite{HNVW16}) since  $L_p(\mathcal{M})$ are UMD spaces. After the appearance of noncommutative harmonic analysis,
it was shown remarkable that any {\emph{standard}} CZ operator $T$ is of noncommutative weak type $(1,1)$ (cf. \cite{JP1})  and is bounded from $L_\infty$ to noncommutative BMO space (cf. \cite{M}), which also deduce the strong $L_p$ results by interpolation. However, the study on boundedness theory of maximal CZ operators acting on operator-valued functions does not make any progress until ten years later (see \cite{HLX}). An immediate difficulty lies in the fact that the classical maximal function of the  form $M=\sup_{i\in I}|f_i|$ no longer exists in general when 
$(f_i)_{i\in I}$ is a family of operators.  Although, the definition of noncommutative weak type $(1,1)$ maximal inequalities exists in the early stage of noncommutative ergodic theory, the formulation of $L_p$ maximal inequalities was not proposed until the vector-valued noncommutative $L_p$-spaces (see Section 2) appeared two decades ago introduced by Pisier \cite{pisier} and Junge \cite{J1} and has been widely used since then.

%Inspired by the work of Pisier \cite{pisier} and Junge \cite{J1}, 

%Given a Calder\'on-Zygmund operator $T$ with {\emph{standard}} kernel $k$. 
For any $\varepsilon>0$, the associated truncated singular integral operator $T_{\varepsilon}$ is defined by
\begin{equation}\label{maximal12}
T_{\varepsilon}f(x) =\int_{|x-y|>\varepsilon} k(x,y) f(y)dy.
\end{equation}
In particular, it is well-known from classical theory (see e.g. \cite[Proposition 4.1.11]{Gra2014}) that there exists a CZ operator $T_0$ and $b\in L_\infty(\mathbb R^d)$ such that
\begin{align}\label{po}
Tf=T_0f+bf
\end{align}
for all $f\in L_2(\mathbb{R}^{d})$.
Here $T_0$ is the weak operator limit in $\mathcal B(L_2(\mathbb{R}^{d}))$ of truncated operators $T_{\varepsilon_j}$ for some subsequence $(\varepsilon_j)_{j\in\N}\subset(0,\infty)$ with $\varepsilon_j\rightarrow0$ as $j\rightarrow\infty$, whose distributional kernel coincides with $k$ on $\mathbb{R}^d \times \mathbb{R}^d \setminus \{(x,x):\;x\in\mathbb R^d\}$.

%the classical CZ theory (see e.g. \cite[Page 33]{St}) state that there is a subsequence $(\varepsilon_j)_{j\in\N}\subset(0,\infty)$ with $\varepsilon_j\rightarrow0$ as $j\rightarrow\infty$ such that for all $f\in L_p(\R^d)$ with $1<p<\infty$
%\begin{align}\label{po}
%	T_{\varepsilon_j}f(x)\rightarrow Tf(x)-a(x)f(x)
%\end{align}
%as $j\rightarrow\infty$ in the sense of weak $L_p$ topology, where $a$ is a bounded measurable function; moreover, the limit $\lim\limits_{j\rightarrow\infty}T_{\varepsilon_j}f(x)$ exists a.e. for any $f\in C_c^\infty(\R^d)$ if and only if (see e.g. \cite[Page 234]{Gra2014})
%\begin{align}\label{pointwise}\lim_{j\rightarrow\infty}\int_{\varepsilon_j<|x-y|<1}k(x,y)dy\ \mbox{exists\ a.e.}
%\end{align}

When the kernel satisfies the $L_2$-H\"ormander condition (see e.g. \cite{CCP,HLX}) which is quite weaker than the Lipchitz regularity, we obtain a weak type $(1,1)$ criterion  under the assumption of a strong type $(p_0,p_0)$ of $(T_\varepsilon)_{\varepsilon>0}$ with $1<p_0<\infty$ (see \cite[Theorem 1.1]{HLX}). As an important application, we obtained the boundedness theory of {\emph{standard}} maximal
CZ operators of convolution type (see \cite[Theorem 1.3]{HLX}) and left the strong type $(p,p)$ estimates of non-convolution case as an open question.

In the present paper, we give the answer of above question.
\begin{thm}\label{thm:maximal}
Let $T$ be a Calder\'on-Zygmund operator with a {\emph{standard}} kernel $k$. Let $T_{\varepsilon}$ be defined as (\ref{maximal12}). Then the following conclusions hold.
\begin{itemize}
\item[(i)] For $1<p<\infty$, $(T_{\varepsilon})_{\varepsilon>0}$ is of strong type $(p,p)$. More precisely, there exists a constant $C_p$ such that for all $f\in L_p(L_{\infty}(\mathbb{R}^{d})\overline{\otimes}\mathcal{M})$,
\begin{align}\label{maximal}
\big\|{\sup_{\varepsilon>0}}^+T_{\varepsilon}f\big\|_p\leq C_p\|f\|_{p}
\end{align}
with $C_p\leq O({p^4}/{(p-1)^3})$.

\item[(ii)]$(T_{\varepsilon})_{\varepsilon>0}$ is of weak type $(1,1)$. More precisely, there exists a constant $C$ such that for all $f\in L_1(L_{\infty}(\mathbb{R}^{d})\overline{\otimes}\mathcal{M})$ and any $\lambda>0$, there is a projection $e\in L_\infty(\mathbb{R}^{d})\overline{\otimes}\mathcal{M}$ satisfying
		$$\forall \varepsilon>0\ \ \|e(T_{\varepsilon}f)e\|_{\infty}\leq\lambda\ \ \mbox{and}\ \ \tau\otimes\int(e^{\perp})\leq C \frac{\|f\|_1}{\lambda}.$$
		
		\item[(iii)] %If $k$ satisfies additionally \begin{align}\label{pointwise}\lim_{j\rightarrow\infty}\int_{\varepsilon_j<|x-y|<1}k(x,y)dy\ \mbox{exists\ a.e.,}
		%\end{align} then 
		For any $f\in L_p(L_{\infty}(\mathbb{R}^{d})\overline{\otimes}\mathcal{M})$ with $1\leq p<\infty$,
		$$T_{\varepsilon_j}f\xrightarrow{b.a.u.}T_0f$$ %\xrightarrow{b.a.u.}Tf\ 
		 %\begin{align}\label{pointwise}\lim_{\varepsilon\rightarrow0}\int_{\varepsilon<|x-y|<1}k(x,y)dy\ \mbox{exists\ a.e.,}
		%\end{align}
	where $(\varepsilon_j)_{j\in\N}$ and $T_0$ are introduced in (\ref{po}). %is the subsequence introduced in (\ref{po}) and $b$ is a bounded function in $L_{\infty}(\mathbb{R}^{d})\overline{\otimes}\mathcal{M}$.
\end{itemize}
\end{thm}
\begin{remark}\label{constant}
	From the corresponding classical result, it is well-known that $C_p\approx p$ as $p\rightarrow\infty$. When $T_\varepsilon$ in \eqref{maximal} is replaced by $M_\varepsilon$, it is known that $C_p\leq C p^2/(p-1)^2$ from the noncommutative Hardy-Littlewood or Doob's maximal inequalities, see e.g. \cite{M, JX}. But at the moment of writing, we are unable to obtain such an order in the present setting. The difficulty can be explained as the lack of a Marcinkiewicz interpolation theorem for a family of general self-adjoint linear operators (see \cite{CaRi}).
\end{remark}

Let us now describe the strategy of the proof of Theorem \ref{thm:maximal}. The deduction of (ii) from (i) is the main result of our previous work \cite[Theorem 1.1]{HLX}. %; moreover, the first part (i) in the case that the kernel is of convolution type---$k(x,y)=k(x-y)$, has been proved there. 
So the novel part of Theorem \ref{thm:maximal} lies in (i) for CZ operators of non-convolution type, which follows from the following noncommutative Cotlar inequality.

\begin{prop}\label{prop:nc cotlar}
Let $T$ be a Calder\'on-Zygmund operator with a {\emph{standard}} kernel $k$ and $T_{\varepsilon}$ be defined as (\ref{maximal12}). Then for any $1<s<\infty$, $1<p<\infty$ and $f\in L_p(L_{\infty}(\mathbb{R}^{d})\overline{\otimes}\mathcal{M})_{+}$
\begin{equation}\label{nc cotlar}
\begin{split}
\big\|{\sup_{\varepsilon>0}}^{+}T_{\varepsilon} f\big\|_p &\leq  C_{d,\gamma,s}\Big(\|{\sup_{\varepsilon>0}}^{+}M_{\varepsilon} (Tf)\big\|_p+\big\|{\sup_{\varepsilon>0}}^{+}\big(M_{\varepsilon} f^{s}\big)^{\frac{1}{s}}\big\|_p\Big),
\end{split}
\end{equation}
where $M_\varepsilon$ is the Hardy-Littlewood averaging operator
\begin{equation}\label{e:20Mr}
M_\varepsilon g(x)=\frac{1}{c_d\varepsilon^{d}}\int_{|x-y|\leq \varepsilon}g(y)dy
\end{equation}
defined for any reasonable operator-valued function $g$.
\end{prop}

As the experts might be aware of that the Cotlar inequality \eqref{nc cotlar} is much weaker than the classical one (see e.g. \cite[Theorem 4.2.4]{{Gra2014}}):
for all $0<s\leq1$,
\begin{align}\label{weak cotlar}
\sup_{\varepsilon>0}|T_\varepsilon f(x)|\leq C_{d,\gamma,s}\Big(\sup_{\varepsilon>0}(M_\varepsilon(|Tf|^s)(x))^{\frac1s}+\sup_{\varepsilon>0}M_\varepsilon (|f|)(x)\Big),
\end{align}
which yields not only the strong $L_p$ estimates but also the weak $L_1$. As it has been explained briefly in \cite{HLX}, a noncommutative version of \eqref{weak cotlar} in the case $s<1$ seems unavailable. And for operators with kernels of convolution type, the inequality \eqref{nc cotlar} in the case $s=1$ has been obtained in \cite{HLX}. Our approach to \eqref{nc cotlar} and thus Theorem \ref{thm:maximal}(i) seems new even in the classical setting.  Due to the noncommutativity and nonconvolution, at the moment of writing, we are unable to decrease $s$ to 1 or below in \eqref{nc cotlar}.

\section{Preliminaries}

\subsection{Noncommutative $L_{p}$-spaces}
Throughout this paper,  $\M$ will denote a semifinite von Neumann algebra equipped with a \emph{n.s.f} trace $\tau$. Let $\mathcal{S_{\M+}}$ be the collection of $x\in\M_{+}$ such that $\tau(\mathrm{supp}x)<\infty$, where $\M_{+}$ denotes the positive part of $\M$ and $\mathrm{supp}x$ is the support of $x$. Let $\mathcal{S}_{\M}$ be the linear span of $\mathcal{S_{\M+}}$. Then $\mathcal{S}_{\M}$ is a $w^{*}$-dense $\ast$-subalgebra of $\M$. Given $1\leq p<\infty$, we set
$$\|x\|_{p}=[\tau(|x|^p)]^{1/p}\ \ x\in\mathcal{S}_{\M},$$
where $|x|=(x^{\ast}x)^{\frac{1}{2}}$. The completion of the norm space $(\mathcal{S}_{\M},\|\cdot\|_{p})$ is denoted by  $L_{p}(\M)$. As usual, we set $L_{\infty}(\M) = \M$ equipped with the operator norm $\|\cdot\|_{\M}$, and the positive part of $L_{p}(\M)$ is written as $L_{p}(\M)_{+}$. We refer the reader to \cite{FK,PX2} for further details on noncommutative $L_p$-spaces.

\subsection{Noncommutative $\ell_{\infty}$-valued $L_{p}$-spaces}
Let $I$ be an index set. 
Given $1\le p\leq\infty$, $L_p(\mathcal {M};\ell_\infty(I))$ is defined as all sequence $x=(x_n)_{n\in I}$ in $L_p(\mathcal {M})$ for which  there are $a,b\in L_{2p}( \mathcal {M})$  and $y=(y_n)_{n\in I}\subset L_\infty(\mathcal {M})$ such that
$$x_n=ay_nb,\ \ \forall\ n\in I.$$
The norm of $x$ in $L_p(\mathcal {M};\ell_\infty(I))$ is defined as
$$\|x\|_{L_p(\mathcal {M};\ell_\infty(I))}=\inf\left\{\big\|a\big\|_{{2p}}\sup_{n\in I}\big\|y_n\big\|_\infty\big\|b\big\|_{{2p}}\right\},$$
where the infimum runs over all possible factorizations as above.
Usually, the norm of $x$ in $L_p(\mathcal {M};\ell_\infty(I))$ is denoted by $\|{\sup_{n\in I}}^+x_n\|_p$, that is $\|{\sup_{n\in I}}^+x_n\|_p\triangleq\|x\|_{L_p(\mathcal {M};\ell_\infty(I))}$. We refer to \cite{J1} for more details. %In particular, when $x=(x_n)_{n\in I}$ consists of self 

The following useful characterization can be found in e.g. \cite[Remark 4.1]{CXY13}.
\begin{remark}\label{rk:MaxFunct}
	Let $x=(x_n)_{n\in I}$ be a sequence of selfadjoint operators in $L_p(\mathcal {M})$. Then $x\in L_p(\mathcal {M};\ell_\infty)$ if and only if there exists $a\in L_{p}(\M)_{+}$
	such that $-a\leq x_n\leq a$ for all $n\in I$. Moreover, in this case,
	$$\|x\|_{L_p(\mathcal {M}; \ell_\infty(I))}=\inf\Big\{\|a\|_p: a\in L_{p}(\M)_{+}\ \mbox{such that}\ -a\leq x_n\leq a
	,\ \forall n\in I\Big\}.$$
\end{remark}

\noindent In practice, we will omit the index set $I$ if there will be no ambiguity.

\subsection{Almost uniform convergence}
The following definition of (bilaterally) almost uniform convergence considered in Lance's work \cite{Lan76}, which presents a suitable replacement of almost everywhere convergence in the noncommutative setting.
\begin{definition}\label{b.a.u.}\rm
	Let $\mathcal {M}$ be a von Neumann algebra equipped with a \emph{n.s.f} trace $\tau$. %Let $x_n,x\in L_0(\M)$.
	\begin{itemize}
			\item[(i)]  The sequence $(x_n)$ is said to converge bilaterally almost uniformly (b.a.u. in short) to $x$ if for any $\varepsilon>0$, there is a projection $e\in\M$
			such that
			$$\tau(e^\bot)<\varepsilon \ \ \  \mbox{and} \ \ \ \lim _{n\rightarrow\infty}\|e(x_n-x)e\|_\infty=0.$$
			
			\item[(ii)] The sequence $(x_n)$ is said to converge almost uniformly (a.u. in short) to $x$ if for any $\varepsilon>0$, there is a projection $e\in\M$
			such that
			$$\tau(e^\bot)<\varepsilon \ \ \  \mbox{and} \ \ \ \lim _{n\rightarrow\infty}\|(x_n-x)e\|_\infty=0.$$
		\end{itemize}
\end{definition}
%Obviously, $x_{n}\xrightarrow{\rm a.u}x$ implies $x_{n}\xrightarrow{\rm b.a.u}x$. Note that in classical probability spaces, Definition \ref{b.a.u.}
%is equivalent to the usual almost everywhere convergence according to Egorov's theorem.

\subsection{Some auxiliary lemmas}
At the end of this section, we present two lemmas.
The first lemma should be well-known to experts, see e.g. \cite[Theorem 4.3]{CXY13}.
\begin{lemma}\label{result of cxy}
Let $\psi$ be an integrable function on $\mathbb{R}^d$  such that $|\psi|$ is radial and radially decreasing.
Let $\psi_\varepsilon(x)=\frac1{\varepsilon^d}\, \psi(\frac x\varepsilon)$ for $x\in\mathbb{R}^d$ and $\varepsilon>0$. Then for $1<p\leq\infty$
 $$\big\|{\sup_{\varepsilon>0}}^+\psi_\varepsilon\ast f\big\|_p\leq C_{d}\big\|{\sup_{\varepsilon>0}}^+M_\varepsilon f\big\|_p,\;\forall f\in L_p(L_{\infty}(\mathbb{R}^{d})\overline{\otimes}\mathcal{M})_+.$$
Here $\|\cdot\|_p$ denotes $\|\cdot\|_{L_p(L_{\infty}(\mathbb{R}^{d})\overline{\otimes}\mathcal{M})}$.
\end{lemma}
%The second lemma can be found in \cite[Page 306, Remark (ii)]{St}.

By applying the operator convexity of $0\leq t\rightarrow t^{p}$ for $1\leq p\leq2$, we can deduce 
%\begin{lemma}\label{holderinequality1}
%Let $(\Omega,\mu)$ be a measure space. For $f\in L_p(L_{\infty}(\Omega)\overline{\otimes}\mathcal{M})$ with $1\leq p<\infty$,
%$$\int_{A}|f(x)|dx\leq\Big(\int_{A}|f(x)|^{p}dx\Big)^{\frac{1}{p}},\ \forall A\subseteq \Omega\ \mbox{and}\ \mu(A)=1.$$
%where $|A|$ means the Lebesgue measure of $A$.
%\end{lemma}
 the operator-valued version of the H\"{o}lder inequality (see e.g. \cite[Lemma 3.5]{M}).
\begin{lemma}\label{holderinequality}
Let $(\Omega,\mu)$ be a measure space. Then for $1< p<\infty$ and let $p^{\prime}$ be the conjugate index of $p$, that is $1/ p^\prime +1/p =1$, we have
\begin{equation}\label{maximal132}
\int_{\Omega}f(x)g(x)d\mu(x)\leq\Big(\int_{\Omega}f(x)^{p}d\mu(x)\Big)^{\frac{1}{p}}
\Big(\int_{\Omega}g(x)^{p^{\prime}}d\mu(x)\Big)^{\frac{1}{p^{\prime}}},
\end{equation}
where $f:\Omega\rightarrow L_{1}(\mathcal{M})+L_{\infty}(\mathcal{M})$ and $g:\Omega\rightarrow \mathbb{C}$ are positive functions such that all members of inequality (\ref{maximal132}) make sense.
\end{lemma}
\begin{proof}
If we set
$$d\nu(x)=\frac{g(x)^{p^\prime}d\mu(x)}{\int_{\Omega}g(x)^{p^{\prime}}d\mu(x)},$$
then $\nu(\Omega)=1$. By applying the operator convexity, we obtain
\begin{align*}
\hskip1pt \int_{\Omega}f(x)g(x)d\mu(x)
=&\ \hskip1pt \int_{\Omega}f(x)g(x)^{1-p^{\prime}}d\nu(x)\int_{\Omega}g(x)^{p^{\prime}}d\mu(x)\\
\leq&\ \Big(\int_{\Omega}f(x)^{p}g(x)^{(1-p^{\prime})p}d\nu(x)\Big)^{\frac{1}{p}}\int_{\Omega}g(x)^{p^{\prime}}d\mu(x)\\
=&\ \hskip1pt \Big(\int_{\Omega}f(x)^{p}d\mu(x)\Big)^{\frac{1}{p}}\Big(\int_{\Omega}g(x)^{p^{\prime}}d\mu(x)\Big)^{\frac{1}{p^{\prime}}},
\end{align*}
which finishes the proof.
\end{proof}

\section{The noncommutative Cotlar inequalities}

In this section, we prove  Proposition \ref{prop:nc cotlar}---the noncommutative Cotlar inequalities. %The proof of following lemma is quite similar to that in the the classical setting as presented in  \cite[Page 306]{St}, so we omit the proof.
\begin{lemma}\label{mainlemma1}
Let $T$ be a CZ operator defined as (\ref{30}) with a {\emph{standard}} kernel $k$. If we set
$$I_{\varepsilon,N}(x)=\int_{\varepsilon<|x-y|<N}k(x,y)dy,$$
then for $1< p<\infty$, %there is a constant $A>0$ so that
$$\int_{|x-x_{0}|<N}|I_{\varepsilon,N}(x)|^{p}dx\leq C_{d,\gamma,p} N^{d}\ \ \mbox{for\ all}\ \varepsilon,\ N\ \mbox{and}\ x_{0}.$$
%where $C_p=\frac{p^2}{p-1}.$
%Suppose that $k(x,y)$ is given for $x\neq y$ and satisfies (\ref{232}). 
%then for $1\leq p<\infty$, there exists a bounded operator $T:L_{p}\rightarrow L_{p}$ such that (\ref{30}) holds if and only if 
\end{lemma}
The proof of this lemma is quite similar as presented in \cite[Page 306]{St}. For the sake of self-completeness, we present a proof.
\begin{proof}
	%The idea is motivated by the classical setting as  presented in  \cite[Page 306]{St}. %
%	Without loss of generality, we may assume that the kernel $k$ is real-valued. Indeed, by writing the kernel $k$ into its real and imaginary parts, we just need to verify each part separately (see e.g. \cite[Lemma 3.1]{HLX}).
	Note that each ball with radius $N$ can be covered by a finite number of balls with radius $N/2$. Hence, it suffices to show
	\begin{equation}\label{maximal2}
		\int_{|x-x_{0}|<N/2}|I_{\varepsilon,N}(x)|^{p}dx\leq C_{d,\gamma,p}N^{d}\ \ \mbox{for\ all}\ \varepsilon,\ N\ \mbox{and}\ x_{0}.
	\end{equation}

From the $L_p$-boundedness of $T$ and the assumptions on $k$, we know that the truncated operators $T_\varepsilon$ are uniformly bounded in the $L_p$ norm (see e.g. \cite[Page 31, Proposition 1]{St}), that is
%we know that $T$ is bounded on $L_p(\mathbb{R}^{d})$ for $1<p<\infty$ (see e.g. \cite[Theorem A]{JP1}) . Then by considering $\overrightarrow{\mathrm{supp}}  \hskip1pt f$ of $f$ in the proof of \cite[Page 31, Proposition 1]{St}, 
$$\sup_{\varepsilon>0}\|T_\varepsilon f\|_p\leq C_{d,\gamma,p}\|f\|_p,\ \forall f\in L_p(\mathbb{R}^{d}).$$

Now we exploit the above fact by comparing $I_{\varepsilon,N}(x)$ and $T_\varepsilon(\chi^{N,x_0})(x)$ whenever $|x-x_{0}|<N/2$, where $\chi^{N,x_0}$ is the characteristic function of the ball $\{y:|y-x_0|<N\}$.

Noting that under the condition $|x-x_{0}|<N/2$, we have
$$\{y:|x-y|<N\}\Delta\{y:|y-x_0|<N\}\subset\{y:N/2<|x-y|<3N/2\}.$$
As a consequence, if $|x-x_{0}|<N/2$,
\begin{align*}
	\hskip1pt \big|I_{\varepsilon,N}(x)-T_\varepsilon(\chi^{N,x_0})(x)\big|
	=&\ \hskip1pt \Big|\int_{\varepsilon<|x-y|<N}k(x,y)dy-\int_{\varepsilon<|x-y|\atop |y-x_0|<N}k(x,y)dy\Big|\\
	\leq&\ \int_{N/2<|x-y|<3N/2}|k(x,y)|dy
	\leq C_d,
\end{align*}
where the last inequality follows from the size condition (\ref{1}) of the kernel. Therefore, applying the above estimate, the uniform $L_p$ boundedness of $T_\varepsilon$ and the triangle inequality, we   get
\begin{align*}
	\hskip1pt \int_{|x-x_{0}|<N/2}|I_{\varepsilon,N}(x)|^{p}dx\leq&\
	C_{d,p}N^d+\|T_\varepsilon(\chi^{N,x_0})\|^p_p\\
	\leq&\ \hskip1pt C_{d,p}N^d+C_{d,\gamma,p}\|\chi^{N,x_0}\|^p_p\leq C_{d,\gamma,p}N^d.
\end{align*}
%since $\tau(e)=1$. On the other hand, 
%$$\int_{|x-x_{0}|<N/2}|I_{\varepsilon,N}(x)|^{p}dx=\int_{|x-x_{0}|<N/2}\|(I_{\varepsilon,N}\otimes e)(x)\|^p_{L_p(\M)}dx.$$
This completes the proof.
\end{proof}

%Before proving Proposition \ref{prop:nc cotlar}, Now we are ready to prove Proposition \ref{prop:nc cotlar}.% {\color{red} As stated in the introduction, here we should first prove and then the theorem}.
With Lemma \ref{mainlemma1}, we can finish the proof of Proposition \ref{prop:nc cotlar}.
\begin{proof}[Proof of Proposition \ref{prop:nc cotlar}.]
%\subsection{The case $2<p<\infty$}
%I suggest to prove separately the case $2<p<\infty$ and the case $1<p\leq2$
Let $\varphi$ be a positive smooth, radially decreasing
function with integral one  supported in the ball $B(0,\frac{1}{2})$. For any $\varepsilon>0$, we will use the notation $\varphi_{\varepsilon}(x)=\frac{1}{\varepsilon^{d}}\varphi(\frac{x}{\varepsilon})$. Without loss of generality, we may assume that the kernel $k$ is real-valued. Indeed, by writing the kernel $k$ into its real and imaginary parts, we just need to verify each part separately (see e.g. \cite[Lemma 3.1]{HLX}).   On the other hand, since 
$C_c^\infty(\R^d)\otimes\mathcal{S}_{\M}$ is dense in $L_p(L_{\infty}(\mathbb{R}^{d})\overline{\otimes}\mathcal{M})$, it suffices to consider any fixed positive function $f\in C_c^\infty(\R^d)\otimes\mathcal{S}_{\M}$. In the following, we fix $1<s<\infty$ and $1<p<\infty$. Decompose $T_{\varepsilon}f$ as
%Fix $f\in L_p(L_{\infty}(\mathbb{R}^{d})\overline{\otimes}\mathcal{M})_+$, we express $T_{\varepsilon}f$ as
\begin{eqnarray}\label{k1}
T_{\varepsilon}f(x)=\varphi_{\varepsilon}\ast(Tf)(x)+
T_{\varepsilon}f(x)-\varphi_{\varepsilon}\ast(Tf)(x).
\end{eqnarray}
The first term of right-hand side of (\ref{k1}) is easy to handle. Indeed, by Lemma \ref{result of cxy}, 
\begin{eqnarray}\label{k2}
\big\|{\sup_{\varepsilon>0}}^+\varphi_{\varepsilon}\ast(Tf)\big\|_p\leq C_{d} \big\|{\sup_{\varepsilon>0}}^{+}M_{\varepsilon} (Tf)\big\|_p.
\end{eqnarray}

On the other hand, recalled in \eqref{po} that there exists a CZ operator $T_0$ such that $Tf=T_0f+bf$,  where $T_0$ is the weak operator limit in $\mathcal B(L_2(L_{\infty}(\mathbb{R}^{d})\overline{\otimes}\mathcal{M}))$ of truncated operators of $T_{\varepsilon_{j}}$ for some sequence $\varepsilon_{j}\rightarrow0$ and $b\in {L_{\infty}(\mathbb{R}^{d})}$. %from (\ref{po}), we have  such that the distributional kernel of $T_0$ coincides with $k$ on $\mathbb{R}^d \times \mathbb{R}^d \setminus \{(x,x):\;x\in\mathbb R^d\}$  and $\|b\|_{L_{\infty}(\mathbb{R}^{d})}\leq C$. 
As a consequence,
%know that there exists a CZ operator $T_0$ such that 
\begin{eqnarray}\label{k7}
\varphi_{\varepsilon}\ast(Tf)(x)=\varphi_{\varepsilon}\ast(T_0f)(x)+\varphi_{\varepsilon}\ast(bf)(x).
\end{eqnarray}
By Lemma \ref{result of cxy} and Lemma \ref{holderinequality},  the second term on the right-hand side of (\ref{k7}) is handled as 
\begin{eqnarray}\label{k8}
	\big\|{\sup_{\varepsilon>0}}^+\varphi_{\varepsilon}\ast(bf)\big\|_p\leq C_{d}\|b\|_\infty \big\|{\sup_{\varepsilon>0}}^{+}M_{\varepsilon} f\big\|_p\leq C_{d} \big\|{\sup_{\varepsilon>0}}^{+}\big(M_{\varepsilon} f^{s}\big)^{\frac{1}{s}}\big\|_p.
\end{eqnarray}

According to (\ref{k1})-(\ref{k8}), it remains to estimate $T_{\varepsilon}f(x)-\varphi_{\varepsilon}\ast(T_0f)(x)$. Notice that the kernel associated to $\varphi_{\varepsilon}\ast(T_0f)(x)$ is
$$\int_{\mathbb{R}^{d}}\varphi_{\varepsilon}(x-z)k(z,y)dz,$$
which should be understood as the principle value and equals to
$$\lim_{\varepsilon_{j}\rightarrow0}\int_{|z-y|>\varepsilon_{j}}\varphi_{\varepsilon}(x-z)k(z,y)dz.$$
%for some sequence $\delta_{j}\rightarrow0$. 
If we set
$K(x,y,\varepsilon)\triangleq k(x,y)
\chi_{|x-y|>\varepsilon}-\int_{\mathbb{R}^{d}}\varphi_{\varepsilon}(x-z)k(z,y)dz,$
then 
$$T_{\varepsilon}f(x)-\varphi_{\varepsilon}\ast(T_0f)(x)=\int_{\mathbb{R}^{d}}K(x,y,\varepsilon)
%\Big(k(x,y)
%\chi_{|x-y|>\varepsilon}-\int_{\mathbb{R}^{d}}\varphi_{\varepsilon}(x-z)k(z,y)dz\Big)
f(y)dy.$$

%To see this,
%we rewrite the difference  as
%$$T_{\varepsilon}f(x)-\varphi_{\varepsilon}\ast(T_0f)(x)=\int_{\mathbb{R}^{d}}K(x,y,\varepsilon)
%\Big(k(x,y)
%\chi_{|x-y|>\varepsilon}-\int_{\mathbb{R}^{d}}\varphi_{\varepsilon}(x-z)k(z,y)dz\Big)
%f(y)dy.$$
%Here 
%$$K(x,y,\varepsilon)\triangleq k(x,y)
%\chi_{|x-y|>\varepsilon}-\int_{\mathbb{R}^{d}}\varphi_{\varepsilon}(x-z)k(z,y)dz,$$
%which should be understood as the principle value and equals to
%$$k(x,y)
%\chi_{|x-y|>\varepsilon}-\lim_{\delta_{j}\rightarrow0}\int_{|x-z|>\delta_{j}}\varphi_{\varepsilon}(x-z)k(z,y)dz\chi_{|x-y|\leq\varepsilon}$$
%for some sequence $\delta_{j}\rightarrow0$.
In the following, we decompose
$$
K(x,y,\varepsilon)=K_{1}(x,y,\varepsilon)+K_{2}(x,y,\varepsilon)
$$
where
$$K_{1}(x,y,\varepsilon)=K(x,y,\varepsilon)\chi_{|x-y|>\varepsilon}\ \ \mbox{and}\ \ K_{2}(x,y,\varepsilon)=K(x,y,\varepsilon)\chi_{|x-y|\leq\varepsilon}.$$
To estimate $K_{1}(x,y,\varepsilon)$, by the property of $\varphi$, we get
$$K_{1}(x,y,\varepsilon)=\int_{\mathbb{R}^{d}}\varphi_{\varepsilon}(x-z)(k(x,y)-k(z,y))dz\chi_{|x-y|>\varepsilon}.$$
The support of $\varphi$ gives $|x-z|<\frac{\varepsilon}{2}<\frac{1}{2}|x-y|$. Hence, the Lipschitz smoothness condition and the fact that $K_{1}(x,y,\varepsilon)$ is real-valued imply
\begin{align*}
K_{1}(x,y,\varepsilon)
\leq C\int_{\mathbb{R}^{d}}\varphi_{\varepsilon}(x-z)\frac{|x-z|^{\gamma}}{|x-y|^{d+\gamma}}dz\chi_{|x-y|>\varepsilon}
\leq C \frac{\varepsilon^{\gamma}}{|x-y|^{d+\gamma}}\chi_{|x-y|>\varepsilon}.
\end{align*}
Thus by above observations, 
\begin{align*}
T_{K,\varepsilon,1}f(x)\triangleq\int_{\mathbb{R}^{d}}K_{1}(x,y,\varepsilon)f(y)dy
\leq&\ C\varepsilon^{\gamma}\sum_{k=0}^{\infty}
\int_{2^{k}\varepsilon<|x-y|<2^{k+1}\varepsilon}\frac{f(y)}{|x-y|^{d+\gamma}}dy
\\
\leq&\ C \sum_{k=0}^{\infty}2^{-k\gamma}\frac{1}{(2^{k+1}\varepsilon)^{d}}
\int_{|x-y|<2^{k+1}\varepsilon}f(y)dy.
\end{align*}
On the other hand, it is obvious to see
$$
-C\sum_{k=0}^{\infty}2^{-k\gamma}\frac{1}{(2^{k+1}\varepsilon)^{d}}
\int_{|x-y|<2^{k+1}\varepsilon}f(y)dy\leq T_{K,\varepsilon,1}f(x).$$
Therefore, by Remark \ref{rk:MaxFunct} and Lemma \ref{holderinequality}, we deduce that%and Remark \ref{rk:MaxFunct}:
\begin{eqnarray}\label{k3}
\big\|{\sup_{\varepsilon>0}}^+T_{K,\varepsilon,1}f\big\|_p
\leq C\sum_{k=0}^{\infty}2^{-k\gamma}\|{\sup_{\varepsilon>0}}^+M_{\varepsilon}f\|_p
\leq C_\gamma\|{\sup_{\varepsilon>0}}^+\big(M_{\varepsilon} f^{s}\big)^{\frac{1}{s}}\|_p.
\end{eqnarray}

Now we turn to estimate $K_{2}(x,y,\varepsilon)$. However, in this case
$$K_{2}(x,y,\varepsilon)=-\lim_{\varepsilon_{j}\rightarrow0}\int_{|z-y|>\varepsilon_{j}}\varphi_{\varepsilon}(x-z)k(z,y)dz\chi_{|x-y|\leq\varepsilon}.$$
%which should be understood as the principle value and equals to
%$$-\lim_{\delta_{j}\rightarrow0}\int_{|x-z|>\delta_{j}}\varphi_{\varepsilon}(x-z)k(z,y)dz\chi_{|x-y|\leq\varepsilon}$$
%for some sequence $\delta_{j}\rightarrow0$. %({\color{red}see e.g. \cite[Section 5.3.4]{Gra2008}MORE DETAILS in the introduction, AND it is cited wrongly}).
In the following, we decompose $K_{2}(x,y,\varepsilon)$ as
\begin{align*}
K_{2}(x,y,\varepsilon) & =-\Big[\int_{|z-y|\geq\frac{\varepsilon}{2}}\varphi_{\varepsilon}(x-z)k(z,y)dz\chi_{|x-y|\leq\varepsilon}\\
       & \ \ \ \ \ \ +\lim_{\varepsilon_{j}\rightarrow0}\int_{\varepsilon_{j}<|z-y|<\frac{\varepsilon}{2}}(\varphi_{\varepsilon}(x-z)-
       \varphi_{\varepsilon}(x-y))k(z,y)dz\chi_{|x-y|\leq\varepsilon}\\
       & \ \ \ \ \ \ +\lim_{\varepsilon_{j}\rightarrow0}\int_{\varepsilon_{j}<|z-y|<\frac{\varepsilon}{2}}
       \varphi_{\varepsilon}(x-y)k(z,y)dz\chi_{|x-y|\leq\varepsilon}\Big]\\
       &\triangleq -\Big[A_{\varepsilon}^{1}(x,y)+A_{\varepsilon}^{2}(x,y)+A_{\varepsilon}^{3}(x,y)\Big].
\end{align*}

To handle $A_{\varepsilon}^{1}(x,y)$, by applying the size condition (\ref{1}), we have
$$A_{\varepsilon}^{1}(x,y)\leq C
\int_{|z-y|\geq\frac{\varepsilon}{2}}
\frac{\varphi_{\varepsilon}(x-z)}{|z-y|^{d}}dz\chi_{|x-y|\leq\varepsilon}
\leq C_d\frac{1}{\varepsilon^{d}}\chi_{|x-y|\leq\varepsilon}.$$
Therefore, 
$$-C_d\frac{1}{\varepsilon^{d}}\int_{|x-y|\leq\varepsilon}f(y)dy\leq
T_{K,\varepsilon,2}^{1}f(x)\triangleq\int_{\mathbb{R}^{d}}A_{\varepsilon}^{1}(x,y)f(y)dy
\leq C_d\frac{1}{\varepsilon^{d}}\int_{|x-y|\leq\varepsilon}f(y)dy,$$
which implies that
\begin{eqnarray}\label{k4}
\big\|{\sup_{\varepsilon>0}}^+T^{1}_{K,\varepsilon,2}f\big\|_p
\leq C_d\|{\sup_{\varepsilon>0}}^+M_{\varepsilon}f\|_p\leq C_d\big\|{\sup_{\varepsilon>0}}^{+}\big(M_{\varepsilon} f^{s}\big)^{\frac{1}{s}}\big\|_p.
\end{eqnarray}

For the second term $A_{\varepsilon}^{2}(x,y)$, using the mean value theorem, the size condition (\ref{1}) of the kernel and the polar coordinates formula, we get
$$A_{\varepsilon}^{2}(x,y)\leq C
\frac{1}{\varepsilon^{d+1}}\int_{|z-y|<\frac{\varepsilon}{2}}
\frac{1}{|z-y|^{d-1}}dz\chi_{|x-y|\leq\varepsilon}
\leq C\frac{1}{\varepsilon^{d}}\chi_{|x-y|\leq\varepsilon}.$$
Hence, we deduce that
$$-C\frac{1}{\varepsilon^{d}}\int_{|x-y|\leq\varepsilon}f(y)dy\leq
T_{K,\varepsilon,2}^{2}f(x)\triangleq\int_{\mathbb{R}^{d}}A_{\varepsilon}^{2}(x,y)f(y)dy
\leq C\frac{1}{\varepsilon^{d}}\int_{|x-y|\leq\varepsilon}f(y)dy.$$
This gives rise to
\begin{eqnarray}\label{k5}
\big\|{\sup_{\varepsilon>0}}^+T^{2}_{K,\varepsilon,2}f\big\|_p
\leq C\|{\sup_{\varepsilon>0}}^+M_{\varepsilon}f\|_p\leq C\big\|{\sup_{\varepsilon>0}}^{+}\big(M_{\varepsilon} f^{s}\big)^{\frac{1}{s}}\big\|_p.
\end{eqnarray}

We now deal with the last term $A_{\varepsilon}^{3}(x,y)$. To see this, we rewrite $A_{\varepsilon}^{3}(x,y)$ as
$$A_{\varepsilon}^{3}(x,y)=\lim_{\varepsilon_{j}\rightarrow0}I_{\varepsilon_{j},\frac{\varepsilon}{2}}(y)
       \varphi_{\varepsilon}(x-y)\chi_{|x-y|\leq\varepsilon},$$
where
$$I_{\varepsilon_{j},\frac{\varepsilon}{2}}(y)=\int_{\varepsilon_{j}<|z-y|<\frac{\varepsilon}{2}}k(z,y)dz.$$
By applying the H\"{o}lder inequality stated in Lemma \ref{holderinequality}, we find
\begin{align*}
T_{K,\varepsilon,2}^{3}f(x)\triangleq&\ \int_{|x-y|\leq\varepsilon}A_{\varepsilon}^{3}(x,y)f(y)dy\leq\int_{|x-y|\leq\varepsilon}|A_{\varepsilon}^{3}(x,y)|f(y)dy\\
\leq&\ \Big(\int_{|x-y|\leq\varepsilon}\lim_{\varepsilon_{j}\rightarrow0}|I_{\varepsilon_{j},\frac{\varepsilon}{2}}(y)|^{s^{\prime}}
\varphi_{\varepsilon}(x-y)dy\Big)^{\frac{1}{s^{\prime}}}
\Big(\int_{|x-y|\leq\varepsilon}
\varphi_{\varepsilon}(x-y)f(y)^{s}dy\Big)^{\frac{1}{s}}.
\end{align*}
Using the
Fatou lemma, the support of $\varphi$ and Lemma \ref{mainlemma1}, we can see that the first term above can be controlled by
$$\Big(\lim_{\varepsilon_{j}\rightarrow0}\int_{|x-y|\leq\frac{\varepsilon}{2}}|I_{\varepsilon_{j},\frac{\varepsilon}{2}}(y)|^{s^{\prime}}
\varphi_{\varepsilon}(x-y)dy\Big)^{\frac{1}{s^{\prime}}}\leq C\Big(\frac{1}{\varepsilon^{d}}\lim_{\varepsilon_{j}\rightarrow0}\int_{|x-y|\leq\frac{\varepsilon}{2}}
|I_{\varepsilon_{j},\frac{\varepsilon}{2}}(y)|^{s^{\prime}}
dy\Big)^{\frac{1}{s^{\prime}}}\leq C_{d,\gamma,s}.$$
%\begin{align*}
%&\quad\Big(\int_{|x-y|\leq\varepsilon}\lim_{\delta_{j}\rightarrow0}|I_{\delta_{j},\frac{\varepsilon}{2}}(y)|^{s^{\prime}}
%|\varphi_{\varepsilon}(x-y)|dy\Big)^{\frac{1}{s^{\prime}}}\\
% & \leq\Big(\lim_{\delta_{j}\rightarrow0}\int_{|x-y|\leq\frac{\varepsilon}{2}}|I_{\delta_{j},\frac{\varepsilon}{2}}(y)|^{s^{\prime}}
%|\varphi_{\varepsilon}(x-y)|dy\Big)^{\frac{1}{s^{\prime}}}\\
%& \lesssim\Big(\frac{1}{\varepsilon^{d}}\lim_{\delta_{j}\rightarrow0}\int_{|x-y|\leq\frac{\varepsilon}{2}}
%|I_{\delta_{j},\frac{\varepsilon}{2}}(y)|^{s^{\prime}}
%dy\Big)^{\frac{1}{s^{\prime}}}\lesssim1.
%\end{align*}
%$$\Big(\int_{|x-y|\leq\varepsilon}\lim_{\delta_{j}\rightarrow0}|I_{\delta_{j},\frac{\varepsilon}{8}}(y)|^{p^{\prime}}
%|\varphi_{\varepsilon}(x-y)|dy\Big)^{\frac{1}{p^{\prime}}}\lesssim
%\Big(\frac{1}{\varepsilon^{d}}\int_{|x-y|\leq\varepsilon}\lim_{\delta_{j}\rightarrow0}
%|I_{\delta_{j},\frac{\varepsilon}{8}}(y)|^{p^{\prime}}
%dy\Big)^{\frac{1}{p^{\prime}}}\lesssim1.$$
On the other hand, it is easy to see that
$$\Big(\int_{|x-y|\leq\varepsilon}
\varphi_{\varepsilon}(x-y)f(y)^{s}dy\Big)^{\frac{1}{s}}\leq
C\Big(\frac{1}{\varepsilon^{d}}\int_{|x-y|\leq\varepsilon}
f(y)^{s}dy\Big)^{\frac{1}{s}}.$$
Furthermore, it is trivial that
$$-C_{d,\gamma,s}
\Big(\frac{1}{\varepsilon^{d}}\int_{|x-y|\leq\varepsilon}
f(y)^{s}dy\Big)^{\frac{1}{s}}\leq T_{K,\varepsilon,2}^{3}f(x).$$
Therefore, by Remark \ref{rk:MaxFunct},
\begin{eqnarray}\label{k6}
\big\|{\sup_{\varepsilon>0}}^+T^{3}_{K,\varepsilon,2}f\big\|_p\leq C_{d,\gamma,s}\big\|{\sup_{\varepsilon>0}}^{+}\big(M_{\varepsilon} f^{s}\big)^{\frac{1}{s}}\big\|_p
\end{eqnarray}

Finally,
putting (\ref{k1})-(\ref{k6}) together with the Minkowski inequality, we  obtain %by  for $1<p<\infty$
%, (\ref{k8}) (\ref{k3}), (\ref{k4}), (\ref{k5}) and  
$$
\big\|{\sup_{\varepsilon>0}}^{+}T_{\varepsilon} f\big\|_p\leq C_{d}\|{\sup_{\varepsilon>0}}^{+}M_{\varepsilon} (Tf)\big\|_p+C_{d,\gamma,s}\big\|{\sup_{\varepsilon>0}}^{+}\big(M_{\varepsilon} f^{s}\big)^{\frac{1}{s}}\big\|_p.
$$
This completes the argument of Proposition \ref{prop:nc cotlar}.
\end{proof}

\section{Proof of Theorem \ref{thm:maximal}}

Now we are at a position to prove Theorem \ref{thm:maximal}.
\begin{proof}[Proof of Theorem \ref{thm:maximal} (i) and (ii).]
As we pointed out previously, combining \cite[Theorem 1.1]{HLX}, it suffices to show the strong type $(p,p)$ estimates in (i). %stated in Theorem \ref{thm:maximal}. %moreover, the pointwise convergence result can be obtained as a byproduct of the maximal
%inequalities through a standard verification (see e.g. \cite[Section 7]{HLX}). 

Now let $1<p<\infty$ and $f\in L_p(L_{\infty}(\mathbb{R}^{d})\overline{\otimes}\mathcal{M}).$ By decomposing $f = f_{1}-f_{2} +i(f_{3}-f_{4})$ with positive $f_{j}$ ($j=1,2,3,4$), we assume that $f$ is positive. By Proposition \ref{prop:nc cotlar}, it suffices to estimate the two terms on the right hand of (\ref{nc cotlar}). To this end, we first use the noncommutative Hardy-Littlewood maximal $L_p$ estimate \cite{M} to deduce that
\begin{align}\label{cotlar norm}
\big\|{\sup_{\varepsilon>0}}^{+}M_{\varepsilon} (Tf)\big\|_p\leq C_{d}\frac{p^2}{(p-1)^2}\|Tf\|_p\leq C_{\gamma,d}\frac{p^4}{(p-1)^3}\|f\|_p,
\end{align}
where the second inequality follows from the strong type $(p,p)$ boundedness of $T$ with bound $C_{\gamma,d}\frac{p^2}{p-1}$ (see \cite[Theorem A]{JP1}). %Hence, $C_p=O\big(\frac{p^4}{(p-1)^3}\big)$.

It remains to estimate the second term appeared on the right of (\ref{nc cotlar}).  Choose $s$ such that $1<s<p$. By applying the noncommutative Hardy-Littlewood maximal inequalities to $f^{s}\in L_{p/s}(L_{\infty}(\mathbb{R}^{d})\overline{\otimes}\mathcal{M})$, there exists a positive operator $F\in L_{p/s}(L_{\infty}(\mathbb{R}^{d})\overline{\otimes}\mathcal{M})$ such that for any $\varepsilon>0$,
$$\|F\|_{p/s}\leq C_{d}\frac{p^{2}}{(p-s)^{2}}\|f^{s}\|_{p/s}\quad\mbox{and}\quad
    M_{\varepsilon}(f^{s})\le F.$$
As a consequence, by the monotone increasing property of $0\leq a\rightarrow a^{t}$ for $0<t<1$, we infer that for any $\varepsilon>0$,
$$\big(M_{\varepsilon}(f^{s})\big)^{^{1/s}}\leq F^{1/s}\ \mbox{and}\ \|F^{1/s}\|_{p}=\|F\|^{1/s}_{p/s}\leq C^{1/s}_{d}\frac{p^{2/s}}{(p-s)^{2/s}}\|f\|_{p}.$$
Therefore, we find
\begin{eqnarray}\label{k676}
\big\|{\sup_{\varepsilon>0}}^{+}\big(M_{\varepsilon} f^{s}\big)^{\frac{1}{s}}\big\|_p\leq\|F^{1/s}\|_{p}\leq C^{1/s}_{d}\frac{p^{2/s}}{(p-s)^{2/s}}\|f\|_{p}.
\end{eqnarray}

Finally, together with (\ref{cotlar norm}) and (\ref{k676}) and letting $s\rightarrow1$,
we see that 
$$\big\|{\sup_{\varepsilon>0}}^+T_{\varepsilon}f\big\|_p\leq C_{d,\gamma}\frac{p^4}{(p-1)^3}\|f\|_{p}.$$
%with $C_p=O\big(\frac{p^4}{(p-1)^3}\big)$. 
This finishes the proof. %Finally, by Proposition \ref{prop:nc cotlar}, we infer that for all $1<p<\infty$,
%$$\big\|{\sup_{\varepsilon>0}}^+T_{\varepsilon}f\big\|_p\lesssim\|f\|_{p}.$$
%This finishes the proof of Theorem \ref{thm:maximal}.
\end{proof}
%{\bf\color{blue}Note: I think the proof of the case $2<p<\infty$ and $1< p\leq2$ are similar, so I suggest putting them together.}{thm:quantum}

\begin{proof}[Proof of Theorem \ref{thm:maximal} (iii).]
Let $f\in L_p(L_{\infty}(\mathbb{R}^{d})\overline{\otimes}\mathcal{M})$ and $(\varepsilon_j)_{j\in\N}$ be the sequence that appeared in (\ref{po}). It is already known from  (\ref{po}) that $T_{\varepsilon_j}f\rightarrow T_0f$ in  the sense of the weak operator topology in  $\mathcal B(L_2(L_{\infty}(\mathbb{R}^{d})\overline{\otimes}\mathcal{M}))$. Hence, we just need to show the b.a.u convergence of $(T_{\varepsilon_j}f)_j$ whenever $j\rightarrow\infty$, since then the limit must be $T_0f$ (see e.g. \cite{CXY13}). Moreover,
by \cite[Proposition 1.3]{CL1}, it is enough to show  for any $\delta>0$, there exists a projection $e\in L_{\infty}(\mathbb{R}^{d})\overline{\otimes}\mathcal{M}$ such that 
\begin{eqnarray}\label{k6761}
\int\otimes\tau(1-e)\leq\delta\ \ \mbox{and}\ \ \|e(T_{\varepsilon_k}f-T_{\varepsilon_\ell}f)e\|_{\infty}\rightarrow0\ \mbox{as}\ k,\ell\rightarrow\infty.
\end{eqnarray}

Fix $\delta>0$. For each $n\geq1$, by the density of $C^\infty_c(\mathbb R^d)\otimes S_{\mathcal M}$ in $L_p(L_{\infty}(\mathbb{R}^{d})\overline{\otimes}\mathcal{M})$, we can find $g_n\in C^\infty_c(\mathbb R^d)\otimes S_{\mathcal M}$ such that $\|f-g_n\|_p^p\leq\frac{\delta}{2^n n^p}$. Then using the fact that $(T_{\varepsilon_{j}})_{j\in\N}$ is of weak type $(p,p)$ ($1\leq p<\infty$), there is a projection $e_{n}\in L_{\infty}(\mathbb{R}^{d})\overline{\otimes}\mathcal{M}$ such that
\begin{equation*}\sup_{j}\|e_{n}T_{\varepsilon_{j}} (f-g_{n})e_{n}\|_{\infty}<\frac1n\ \mbox{and}\
	\int\otimes\tau(1-e_n)<n^p\|f-g_n\|_p^p\leq\frac{\delta}{2^{n}}.
\end{equation*}
Let $e=\wedge_{n}e_{n}$. Then 
\begin{align}\label{err}\int\otimes\tau(e^\perp)<\delta\ \mbox{and}\ \sup_{j}\|eT_{\varepsilon_{j}} (f-g_{n})e\|_{\infty}<\frac{1}{n},\ \forall\ n\geq1.\end{align}

On the other hand, we have
\begin{align}\label{uni}
	\lim_{\ell,k\rightarrow+\infty}\|T_{\varepsilon_{\ell}}g_{n}-T_{\varepsilon_{k}}g_{n}\|_{\infty}=0\ \forall n\geq1.
\end{align}
Indeed, write $g_n=\sum_i\phi_{i}\otimes m_{i}$ with finite sum of $i$ such that $\phi_{i}\in C^\infty_c(\mathbb R^d)$ and $m_{i}\in S_{\mathcal M}$. Then for $\varepsilon_\ell<\varepsilon_k$, one has
%by applying the cancellation condition (\ref{pointwise}) of the kernel, we get
\begin{align*}
	T_{\varepsilon_{\ell}}g_{n}(x)-T_{\varepsilon_{k}}g_{n}(x) & = \sum_{i}\int_{\varepsilon_{\ell}<|x-y|\leq\varepsilon_{k}}k(x,y)\phi_{i}(y)dy\otimes m_{i},%\\
	%&=\sum_{i}\int_{\varepsilon_{\ell}<|x-y|\leq\varepsilon_{k}}k(x,y)\big(\phi_{i}(y)-\phi_{i}(x)\big)dy\otimes m_{i}\\
	%&\ \ \ + \sum_{i}\int_{\varepsilon_{\ell}<|x-y|\leq\varepsilon_{k}}k(x,y)\phi_{i}(x)dy\otimes m_{i},
\end{align*}
which tends to 0 as $\ell,k\rightarrow+\infty$ (see e.g. \cite[Page 234]{Gra2014}). Therefore, by (\ref{uni})
we know that for any $n\geq1$, there is a positive constant $N_{n}$ such that for any $\ell,k>N_{n}$
$$\|T_{\varepsilon_{\ell}}g_{n}-T_{\varepsilon_{k}}g_{n}\|_{\infty}<\frac{1}{n}.$$
Then together with \eqref{err}, we finally get for any $\ell,k>N_{n}$,
\begin{align*}
	\|e(T_{\varepsilon_{\ell}}f-T_{\varepsilon_{k}}f)e\|_{\infty} &\leq \|T_{\varepsilon_{\ell}}g_{n}-T_{\varepsilon_{k}}g_{n}\|_{\infty}+\|e
	T_{\varepsilon_{\ell}}(f-g_{n})e\|_{\infty}
	+\|eT_{\varepsilon_{k}}(f-g_n)e\|_{\infty}<\frac{3}{n},
\end{align*}
which yields
$$\lim_{\ell,k\rightarrow\infty}\|e(T_{\varepsilon_{\ell}}f-T_{\varepsilon_{k}}f)e\|_{\infty}=0.$$
%Therefore, by Lemma \ref{CL1}, $(T_{\varepsilon_{j}}f)_{j\in\N}$ converges b.a.u. to $Tf(\triangleq\lim_{j\rightarrow+\infty}T_{\varepsilon_{j}}f)$, that is, $Tf$ exists in the sense of the principal value.
This completes the proof.
\end{proof}

%\subsection*{Acknowledgement} The third named author acknowledges the DST-INSPIRE Faculty Fellowship DST/INSPIRE/04/2020/001132 for financial support.

\end{document}